\documentclass{birkjour}
\usepackage{hyperref,color}
\usepackage{amsmath,amsfonts,amssymb}

 \newtheorem{theorem}{Theorem}[section]
 \newtheorem{corollary}[theorem]{Corollary}
 \newtheorem{lemma}[theorem]{Lemma}
 \numberwithin{equation}{section}
\newcommand{\C}{\mathbb{C}}
\newcommand{\N}{\mathbb{N}}
\newcommand{\R}{\mathbb{R}}
\newcommand{\cA}{\mathcal{A}}
\newcommand{\cB}{\mathcal{B}}

\newcommand{\cF}{\mathcal{F}}
\newcommand{\cK}{\mathcal{K}}
\newcommand{\cM}{\mathcal{M}}
\newcommand{\cS}{\mathcal{S}}
\newcommand{\fA}{\mathfrak{A}}
\newcommand{\fB}{\mathfrak{B}}
\newcommand{\fC}{\mathfrak{C}}

\newcommand{\fM}{\mathfrak{M}}
\newcommand{\fS}{\mathfrak{S}}
\newcommand{\alg}{\operatorname{alg}}
\newcommand{\spe}{\operatorname{sp}}
\newcommand{\slim}{\operatornamewithlimits{s-\!\lim}_{n\to\infty}}
\newcommand{\eps}{\varepsilon}
\newcommand{\dR}{{\bf\dot{\R}}}
\begin{document}
\title[Calkin images of Fourier convolution operators]%
{{Calkin images of Fourier convolution\\
operators with slowly oscillating symbols}}
\author{C. A. Fernandes} 
\address{%
Centro de Matem\'atica e Aplica\c{c}\~oes,\\
Departamento de Matem\'a\-tica, \\
Faculdade de Ci\^encias e Tecnologia,\\
Universidade Nova de Lisboa,\\
Quinta da Torre, \\
2829--516 Caparica, Portugal}
\email{caf@fct.unl.pt} 
\author{A. Yu. Karlovich}
\address{%
Centro de Matem\'atica e Aplica\c{c}\~oes,\\
Departamento de Matem\'a\-tica, \\
Faculdade de Ci\^encias e Tecnologia,\\
Universidade Nova de Lisboa,\\
Quinta da Torre, \\
2829--516 Caparica, Portugal}
\email{oyk@fct.unl.pt}
\author{Yu. I. Karlovich}
\address{%
Centro de Investigaci\'on en Ciencias,\\
Instituto de Investigaci\'on en Ciencias B\'asicas y Aplicadas,\\
Universidad Aut\'onoma del Estado de Morelos,\\
Av. Universidad 1001, Col. Chamilpa,\\
C.P. 62209 Cuernavaca, Morelos, M\'exico}
\email{karlovich@uaem.mx}
\thanks{%
This work was partially supported by the Funda\c{c}\~ao para a Ci\^encia e a
Tecnologia (Portu\-guese Foundation for Science and Technology) through the
project UID/MAT/00297/2019 (Centro de Matem\'atica e Aplica\c{c}\~oes). The
third author was also supported by the SEP-CONACYT Project A1-S-8793
(M\'exico).}
\begin{abstract}
Let $\Phi$ be a $C^*$-subalgebra of $L^\infty(\R)$
{
and $SO_{X(\R)}^\diamond$ be the Banach algebra of slowly
oscillating Fourier multipliers on a Banach function space $X(\R)$.}
 We show that the intersection of the Calkin
image of the algebra generated by the operators of multiplication $aI$ by
functions $a\in\Phi$ and the Calkin image of the algebra generated by the
Fourier convolution operators $W^0(b)$ with symbols in
{$SO_{X(\R)}^\diamond$}
coincides with the Calkin image of the algebra
generated by the operators of multiplication by constants.
\end{abstract}
\maketitle
\section{Introduction}
Let $\cF:L^2(\R)\to L^2(\R)$ denote the Fourier transform
\[
(\cF f)(x):=\widehat{f}(x):=\int_\R f(t)e^{itx}\,dt,
\quad
x\in\R,
\]
and let $\cF^{-1}:L^2(\R)\to L^2(\R)$ be the inverse of $\cF$,
\[
(\cF^{-1}g)(t)=\frac{1}{2\pi}\int_\R g(x)e^{-itx}\,d x,
\quad
t\in\R.
\]
It is well known that the Fourier convolution operator
\begin{equation}\label{eq:convolution-operator}
W^0(a):=\cF^{-1}a\cF
\end{equation}
is bounded on the space $L^2(\R)$ for every $a\in L^\infty(\R)$.

Let $X(\R)$ be a Banach function space and $X'(\R)$ be its associate
space. Their technical definitions are postponed to
Section~\ref{sec:BFS}. The class of Banach function spaces is
very large. It includes Lebesgue, Orlicz, Lorentz spaces, variable
Lebesgue spaces and their weighted analogues (see, e.g., \cite{BS88,CF13}).
Let $\cB(X(\R))$ denote the Banach algebra of all bounded linear operators
acting on $X(\R)$, let $\cK(X(\R))$ be the closed two-sided ideal
of all compact operators in $\cB(X(\R))$, and let
$\cB^\pi(X(\R))=\cB(X(\R))/\cK(X(\R))$ be the Calkin algebra of the
cosets $A^\pi:=A+\cK(X(\R))$, where $A\in\cB(X(\R))$.

If $X(\R)$ is separable, then $L^2(\R)\cap X(\R)$
is dense in $X(\R)$ (see Lemma~\ref{le:density-Cc-infty} below). A function
$a\in L^\infty(\R)$ is called a Fourier multiplier on $X(\R)$ if the
convolution operator $W^0(a)$ defined by \eqref{eq:convolution-operator}
maps $L^2(\R)\cap X(\R)$ into $X(\R)$ and extends to a bounded linear
operator on $X(\R)$. The function $a$ is called the symbol of the Fourier
convolution operator $W^0(a)$. The set $\cM_{X(\R)}$ of all Fourier
multipliers on  $X(\R)$ is a unital normed algebra under pointwise
operations and the norm
\[
\left\|a\right\|_{\cM_{X(\R)}}:=\left\|W^0(a)\right\|_{\cB(X(\R))}.
\]

For a unital $C^*$-subalgeb\-ra $\Phi$ of the algebra $L^\infty(\R)$, we
consider the quotient algebra $\mathcal{MO}^\pi(\Phi)$ consisting of the
cosets
\[
[aI]^\pi:=aI+\cK(X(\R))
\]
of multiplication operators by functions in $\Phi$:
\[
\mathcal{MO}^\pi(\Phi)
:=
\{[aI]^\pi: \ a\in\Phi\}
=
\{aI+\cK(X(\R)):\ a\in\Phi\}.
\]
For a unital Banach subalgebra $\Psi$ of the algebra $\cM_{X(\R)}$, we also
consider the  quotient algebra $\mathcal{CO}^\pi(\Psi)$ consisting of the
cosets
\[
[W^0(b)]^\pi:=W^0(b)+\cK(X(\R))
\]
of convolution operators with symbols
in the algebra $\Psi$:
\[
\mathcal{CO}^\pi(\Psi)
:=
\{[W^0(b)]^\pi: \ b\in\Psi\}
=
\{W^0(b)+\cK(X(\R)) :\ b\in\Psi\}.
\]

It is easy to see that $\mathcal{MO}^\pi(\Phi)$ and
$\mathcal{CO}^\pi(\Psi)$ are commutative unital
Banach subalgebras of the Calkin algebra $\cB^\pi(X(\R))$.
It is natural to refer to the algebras $\mathcal{MO}^\pi(\Phi)$
and $\mathcal{CO}^\pi(\Psi)$ as the Calkin images of the algebras
\[
\mathcal{MO}(\Phi)=\{aI:a\in\Phi\}\subset\cB(X(\R)),
\
\mathcal{CO}(\Psi)=\{W^0(b):b\in\Psi\}\subset\cB(X(\R)),
\]
respectively. The algebras $\mathcal{MO}(\Phi)$ and
$\mathcal{CO}(\Psi)$ are building blocks of the algebra of convolution
type operators
\[
\cA(\Phi,\Psi;X(\R))=
\alg_{\cB(X(\R))}\big\{aI,W^0(b):\ a\in\Phi,\ b\in\Psi\big\},
\]
the smallest closed subalgebra of $\cB(X(\R))$ that contains the
algebras $\mathcal{MO}(\Phi)$ and $\mathcal{CO}(\Psi)$.

Let $SO^\diamond$ be the $C^*$-algebra of slowly oscillating functions
and $SO_{X(\R)}^\diamond$ be the Banach algebra of all slowly oscillating
Fourier multipliers on the space $X(\R)$, which are defined below in
Sections~\ref{sec:SO-functions}--\ref{sec:SO-multipliers}.
The third author proved in \cite[Lemma~4.3]{K17} in the case
of Lebesgue spaces $L^p(\R,w)$, $1<p<\infty$, with Muckenhoupt weights
$w\in A_p(\R)$ that
\begin{equation}\label{eq:old-equality}
\mathcal{MO}^\pi(SO^\diamond)\cap\mathcal{CO}^\pi(SO_{L^p(\R,w)}^\diamond)
=
\mathcal{MO}^\pi(\C),
\end{equation}
where
\begin{equation}\label{eq:multiplications-by-constants}
\mathcal{MO}^\pi(\C):=\{[cI]^\pi:\ c\in\C\}.
\end{equation}
This result allowed him to describe the maximal ideal space
of the commutative Banach algebra
\[
\cA^\pi(SO^\diamond,SO^\diamond_{L^p(\R,w)};L^p(\R,w))
=\cA(SO^\diamond,SO^\diamond_{L^p(\R,w)};L^p(\R,w))/\cK(L^p(\R,w))
\]
(see \cite[Theorem~3.1]{K17}). In turn, this description plays a crucial role
in the study of {the Fredholmness} of operators in more general algebras
of convolution type operators with piecewise slowly oscillating data
on weighted Lebesgue space $L^p(\R,w)$
(see \cite{K17,KILH12,KILH13a}).

Recall that the {(non-centered)} Hardy-Littlewood maximal function $\cM f$
of a function $f\in L_{\rm loc}^1(\R)$ is defined by
\[
(\cM f)(x):=\sup_{I\ni x}\frac{1}{|I|}\int_I|f(y)|\,dy,
\]
where the supremum is taken over all intervals $I\subset\R$ of finite length
containing $x$. The Hardy-Littlewood maximal operator $\cM$ defined by the rule
$f\mapsto \cM f$ is a sublinear operator.

The aim of this paper is to extend \eqref{eq:old-equality} to the case of
separable Banach function spaces such that the Hardy-Littlewood maximal
operator $\cM$ is bounded on $X(\R)$ and on its associate space $X'(\R)$
{and to}
the case of arbitrary algebras of functions $\Phi\subset L^\infty(\R)$
in place of {$SO^\diamond$.}

The following statement extends \cite[Lemma~4.3]{K17}.
\begin{theorem}[Main result]
\label{th:intersection-quotient-algebras}
Let $X(\R)$ be a separable Banach function space such that the Hardy-Littlewood
maximal operator $\cM$ is bounded on the space $X(\R)$ and on its associate
space $X'(\R)$. If $\Phi$ is a unital $C^*$-subalgebra of
{$L^\infty(\R)$, then}
\begin{equation}\label{eq:intersection-quotient-algebras-1}
\mathcal{MO}^\pi(\Phi)
\cap
\mathcal{CO}^\pi{(SO_{X(\R)}^\diamond)}
=\mathcal{MO}^\pi(\C),
\end{equation}
where $\mathcal{MO}^\pi(\C)$ is defined by
\eqref{eq:multiplications-by-constants}.
\end{theorem}
This result is one more step towards the study of Fredholm properties of
convolution type operators with discontinuous data on Banach function spaces
more general than weighted Lebesgue spaces initiated in the authors works
\cite{FKK-AFA,FKK-FS12,FKK-ISAAC19}.

One can expect, by analogy with the case of weighted Lebesgue spaces,
that, for instance,
$\cK(X(\R))\subset\cA(SO^\diamond,SO_{X(\R)}^\diamond;X(\R))$
and that the quotient algebra
\[
\cA^\pi(SO^\diamond,SO_{X(\R)}^\diamond;X(\R))
=
\cA(SO^\diamond,SO_{X(\R)}^\diamond;X(\R))/\cK(X(\R))
\]
is commutative. It seems, however, that the proofs of both hypotheses will
require tools, which are not available in the setting of general Banach
function spaces. We plan to return to these questions in a forthcoming
work, restricting ourselves to particular Banach function spaces, like
rearrangement-invariant spaces with Muckenhoupt weights or variable Lebesgue
spaces, where interpolation theorems are available.

The paper is organized as follows.
In Section~\ref{sec:preliminaries}, we collect necessary facts on Banach
function spaces and Fourier multipliers on them. Further, we recall the
definition of the $C^*$-algebra $SO^\diamond$ of slowly oscillating functions
and introduce the Banach algebra of slowly oscillating Fourier multipliers
$SO_{X(\R)}^\diamond$ on a Banach function spaces $X(\R)$.
In Section~\ref{sec:maximal-ideal-space}, we discuss the structure of the
maximal ideal
spaces $M(SO^\diamond)$ and $M(SO_{X(\R)}^\diamond)$ of the $C^*$-algebra
$SO^\diamond$ of slowly oscillating functions and the Banach algebra
$SO_{X(\R)}^\diamond$ of slowly oscillating Fourier multipliers
on a Banach function space $X(\R)$.
In particular, we show that the fibers $M_t(SO^\diamond)$
of $M(SO^\diamond)$ over the points $t\in\dR:=\R\cup\{\infty\}$
can be identified with the fibers $M_t(SO_t)$, where $SO_t$ is the
$C^*$-algebra of all bounded continuous functions on
$\dR\setminus\{t\}$ that slowly oscillate at the point $t$.
An analogous result is also obtained for the fibers of the maximal
ideal spaces of algebras of slowly oscillating
Fourier multipliers on a Banach function space $X(\R)$.
In Section~\ref{sec:maximal-ideal-spaces-Calkin-images}, we show that
the maximal ideal spaces of the algebras $\mathcal{MO}^\pi(\Phi)$
and $\mathcal{CO}^\pi(\Psi)$ are homeomorphic to the maximal ideal
spaces of the algebras $\Phi$ and $\Psi$, respectively,
where $\Phi$ is a unital $C^*$-subalgebra of $L^\infty(\R)$ and
$\Psi$ is a unital Banach subalgebra of $\cM_{X(\R)}$.
In Section~\ref{sec:proofs}, we recall the definition of a limit operator
(see \cite{RRS04} for a general theory of limit operators), as well as,
{a known fact about limit operators of compact operators
acting} on Banach function spaces. Further, we calculate the limit operators
of the Fourier convolution operator $W^0(b)$ with a slowly oscillating symbol
$b\in SO_{X(\R)}^\diamond$. Finally, gathering the above mentioned results
on limit operators, we prove Theorem~\ref{th:intersection-quotient-algebras}.
\section{Preliminaries}\label{sec:preliminaries}
\subsection{Banach function spaces}\label{sec:BFS}
The set of all Lebesgue measurable complex-valued functions on $\R$ is denoted
by $\fM(\R)$. Let $\fM^+(\R)$ be the subset of functions in $\fM(\R)$ whose
values lie  in $[0,\infty]$. The Lebesgue measure of a measurable set
$E\subset\R$ is denoted by $|E|$ and its characteristic function is denoted
by $\chi_E$. Following \cite[Chap.~1, Definition~1.1]{BS88}, a mapping
$\rho:\fM^+(\R)\to [0,\infty]$ is called a Banach function norm if,
for all functions $f,g, f_n \ (n\in\N)$ in $\fM^+(\R)$, for all
constants $a\ge 0$, and for all measurable subsets $E$ of $\R$, the following
properties hold:
\begin{eqnarray*}
{\rm (A1)} & & \rho(f)=0  \Leftrightarrow  f=0\ \mbox{a.e.}, \quad
\rho(af)=a\rho(f), \quad
\rho(f+g) \le \rho(f)+\rho(g),\\
{\rm (A2)} & &0\le g \le f \ \mbox{a.e.} \ \Rightarrow \ \rho(g)
\le \rho(f)
\quad\mbox{(the lattice property)},
\\
{\rm (A3)} & &0\le f_n \uparrow f \ \mbox{a.e.} \ \Rightarrow \
       \rho(f_n) \uparrow \rho(f)\quad\mbox{(the Fatou property)},\\
{\rm (A4)} & & |E|<\infty \Rightarrow \rho(\chi_E) <\infty,\\
{\rm (A5)} & & |E|<\infty \Rightarrow \int_E f(x)\,dx \le C_E\rho(f)
\end{eqnarray*}
with $C_E \in (0,\infty)$ which may depend on $E$ and $\rho$ but is
independent of $f$. When functions differing only on a set of measure zero
are identified, the set $X(\R)$ of all functions $f\in\fM(\R)$
for which $\rho(|f|)<\infty$ is called a Banach function space. For each
$f\in X(\R)$,
the norm of $f$ is defined by
\[
\left\|f\right\|_{X(\R)} :=\rho(|f|).
\]
Under the natural linear space operations and under this norm, the set
$X(\R)$ becomes a Banach space (see \cite[Chap.~1, Theorems~1.4 and~1.6]{BS88}).
If $\rho$ is a Banach function norm, its associate norm $\rho'$ is
defined on $\fM^+(\R)$ by
\[
\rho'(g):=\sup\left\{
\int_{\R} f(x)g(x)\,dx \ : \ f\in \fM^+(\R), \ \rho(f) \le 1
\right\}, \quad g\in \fM^+(\R).
\]
It is a Banach function norm itself \cite[Chap.~1, Theorem~2.2]{BS88}.
The Banach function space $X'(\R)$ determined by the Banach function norm
$\rho'$ is called the associate space (K\"othe dual) of $X(\R)$.
The associate space $X'(\R)$ is naturally identified with a subspace
of the (Banach) dual space $[X(\R)]^*$.
\subsection{Density of nice functions in separable Banach function spaces}
As usual, let $C_0^\infty(\R)$ denote the set of all infinitely
differentiable functions with compact support.
\begin{lemma}[{\cite[Lemma~2.1]{FKK-AFA} and \cite[Lemma~2.12(a)]{KS14}}]
\label{le:density-Cc-infty}
If $X(\R)$ is a separable Banach function space, then the sets $C_0^\infty(\R)$
and $L^2(\R)\cap X(\R)$ are dense in the space $X(\R)$.
\end{lemma}
Let $\cS(\R)$ be the Schwartz space of rapidly decreasing smooth functions
and let $\cS_0(\R)$ denote the set of functions $f\in\cS(\R)$ such that
their Fourier transforms $\cF f$ have compact support.
\begin{theorem}[{\cite[Theorem~4]{FKK-ISAAC19}}]
\label{th:density-S0}
Let $X(\R)$ be a separable Banach function space such that the Hardy-Littlewood
maximal operator $\cM$ is bounded on $X(\R)$. Then the set $\cS_0(\R)$ is dense
in the space $X(\R)$.
\end{theorem}
\subsection{Banach algebra \boldmath{$\cM_{X(\R)}$} of Fourier multipliers}
The following result plays an important role in this paper.
\begin{theorem}[{\cite[Corollary~4.2]{KS19} and \cite[Theorem~2.4]{FKK-AFA}}]
\label{th:continuous-embedding}
Let $X(\R)$ be a separable Banach function space such that the
Hardy-Littlewood maximal operator $\cM$ is bounded on $X(\R)$ and on its
associate space $X'(\R)$. If $a\in\cM_{X(\R)}$, then
\begin{equation}\label{eq:continuous-embedding}
\|a\|_{L^\infty(\R)}\le\|a\|_{\cM_{X(\R)}}.
\end{equation}
The constant $1$ on the right-hand side of \eqref{eq:continuous-embedding}
is best possible.
\end{theorem}
Inequality \eqref{eq:continuous-embedding} was established earlier in
\cite[Theorem~1]{K15b} with some constant on the right-hand side
that depends on the space $X(\R)$.

Since \eqref{eq:continuous-embedding} is available, an easy adaptation of
the proof of \cite[Proposition~2.5.13]{G14} leads to the following
(we refer to the proof of \cite[Corollary~1]{K15b} for details).
\begin{corollary}
Let $X(\R)$ be a separable Banach function space such that the
Hardy-Littlewood maximal operator $\cM$ is bounded on $X(\R)$ and on its
associate space $X'(\R)$. Then the set of Fourier multipliers
$\cM_{X(\R)}$ is a Banach algebra under pointwise operations and the norm
$\|\cdot\|_{\cM_{X(\R)}}$.
\end{corollary}
\subsection{Stechkin-type inequality}
Let $V(\R)$ be the Banach algebra of all functions $a:\R\to\C$ with finite
total variation
\[
V(a):=\sup\sum_{i=1}^n|a(t_i)-a(t_{i-1})|,
\]
where the supremum is taken over all finite partitions
\[
-\infty<t_0<t_1<\dots<t_n<+\infty
\]
of the real line $\R$ and the norm in $V(\R)$ is given by
\[
\|a\|_{V}=\|a\|_{L^\infty(\R)}+V(a).
\]
\begin{theorem}\label{th:Stechkin}
Let $X(\R)$ be a separable Banach function space such that the
Hardy-Littlewood maximal operator $\cM$ is bounded on $X(\R)$ and on its
associate space $X'(\R)$. If $a\in V(\R)$, then the convolution operator
$W^0(a)$ is bounded on the space $X(\R)$ and
\begin{equation}\label{eq:Stechkin}
\|W^0(a)\|_{\cB(X(\R))}
\le
c_{X}\|a\|_V
\end{equation}
where $c_{X}$ is a positive constant depending only on $X(\R)$.
\end{theorem}
This result follows from \cite[Theorem~4.3]{K15a}.

For Lebesgue spaces $L^p(\R)$, $1<p<\infty$, inequality~\eqref{eq:Stechkin} is
usually called Stechkin's inequality, and the constant $c_{L^p}$ is
calculated explicitly:
\begin{equation}\label{eq:constant-in-Stechkin}
c_{L^p}=\|S\|_{\cB(L^p(\R))}=\left\{\begin{array}{ccc}
\tan\left(\frac{\pi}{2p}\right) &\mbox{if}& 1<p\le 2,
\\[3mm]
\cot\left(\frac{\pi}{2p}\right) &\mbox{if}& 2\le p<\infty,
\end{array}\right.
\end{equation}
where $S$ is the Cauchy singular integral operator given by
\[
(Sf)(x):=\frac{1}{\pi i}\lim_{\eps\to 0}\int_{\R\setminus(x-\eps,x+\eps)}
\frac{f(t)}{t-x}\,dt.
\]
We refer to \cite[Theorem~2.11]{D79} for the proof of \eqref{eq:Stechkin}
in the case of Lebesgue spaces $L^p(\R)$ with $c_{L^p}=\|S\|_{\cB(L^p(\R))}$
and to \cite[Chap. 13, Theorem 1.3]{GK92} for the calculation of the norm
of $S$ given in the second equality in \eqref{eq:constant-in-Stechkin}.
For Lebesgue spaces with Muckenhoupt weights $L^p(\R,w)$, the proof of
Theorem~\ref{th:Stechkin} with $c_{L^p(w)}=\|S\|_{\cB(L^p(\R,w))}$ is
contained in \cite[Theorem~17.1]{BKS02}. Further, for variable Lebesgue
spaces $L^{p(\cdot)}(\R)$, Theorem~\ref{th:Stechkin} with
$c_{L^{p(\cdot)}}=\|S\|_{\cB(L^{p(\cdot)}(\R))}$ was obtained in
\cite[Theorem~2]{K15d}.
\subsection{Slowly oscillating functions}\label{sec:SO-functions}
Let $\dR=\R\cup\{\infty\}$. For a set $E\subset\dR$ and a function
$f:\dR\to\C$ in $L^\infty(\R)$, let the oscillation of
$f$ over $E$ be defined by
\[
\operatorname{osc}(f,E)
:=
\operatornamewithlimits{ess\,sup}_{s,t\in E}|f(s)-f(t)|.
\]
Following \cite[Section~4]{BFK06} and
\cite[Section~2.1]{KILH12}, \cite[Section~2.1]{KILH13a},
we say that a function $f\in L^\infty(\R)$ is slowly
oscillating at a point $\lambda\in\dR$ if for every $r\in(0,1)$ or,
equivalently, for some $r\in(0,1)$, one has
\[
\begin{array}{lll}
\lim\limits_{x\to 0+}
\operatorname{osc}\big(f,\lambda+([-x,-rx]\cup[rx,x])\big)=0
&\mbox{if}& \lambda\in\R,
\\
\lim\limits_{x\to +\infty}
\operatorname{osc}\big(f,[-x,-rx]\cup[rx,x]\big)=0
&\mbox{if}& \lambda=\infty.
\end{array}
\]
For every $\lambda\in\dR$, let $SO_\lambda$ denote the $C^*$-subalgebra of
$L^\infty(\R)$ defined by
\[
SO_\lambda:=\left\{f\in C_b(\dR\setminus\{\lambda\})\ :\ f
\mbox{ slowly oscillates at }\lambda\right\},
\]
where
$C_b(\dR\setminus\{\lambda\}):=C(\dR\setminus\{\lambda\})\cap L^\infty(\R)$.

Let $SO^\diamond$ be the smallest $C^*$-subalgebra of $L^\infty(\R)$ that
contains all the $C^*$-algebras $SO_\lambda$ with $\lambda\in\dR$.
The functions in $SO^\diamond$ are called slowly oscillating functions.
\subsection{Banach \boldmath{$SO_\lambda^3$} of three times continuously
differentiable slowly oscillating functions}
For a point $\lambda\in\dR$, let $C^3(\R\setminus\{\lambda\})$ be the set of
all three times continuously differentiable functions
$a:\R\setminus\{\lambda\}\to\C$.
Following \cite[Section~2.4]{KILH12} and \cite[Section~2.3]{KILH13a}, consider
the commutative Banach algebras
\[
SO_\lambda^3:=\left\{
a\in SO_\lambda\cap C^3(\R\setminus\{\lambda\})\ :\
\lim_{x\to\lambda}(D_\lambda^k a)(x)=0,
\ k=1,2,3
\right\}
\]
equipped with the norm
\[
\|a\|_{SO_\lambda^3}:=
\sum_{j=0}^3\frac{1}{j!}\left\|D_\lambda^ka\right\|_{L^\infty(\R)},
\]
where $(D_\lambda a)(x)=(x-\lambda) a'(x)$ for $\lambda\in\R$ and
$(D_\lambda a)(x)=xa'(x)$ for $\lambda=\infty$.
\begin{lemma}\label{le:SO3-lambda-is-dense-in-SO-lambda}
For every $\lambda\in\dR$, the set $SO_\lambda^3$ is dense in the
$C^*$-algebra $SO_\lambda$.
\end{lemma}
\begin{proof}
In view of \cite[Lemma~2.3]{BBK04}, the set
\begin{equation}\label{eq:SO3-lambda-is-dense-in-SO-lambda-0}
SO_\infty^\infty:=
\left\{f\in SO_\infty\cap C_b^\infty(\R)\ :\
\lim_{x\to\infty}(D_\infty^kf)(x)=0,\ k\in\N
\right\}
\end{equation}
is dense in the Banach algebra $SO_\infty$. Here $C_b^\infty(\R)$ denotes the
set of all infinitely differentiable functions $f:\R\to\C$, which are bounded
with all their derivatives. Note that $SO_\infty^\infty$ can be
equivalently defined by replacing $C_b^\infty$ in
\eqref{eq:SO3-lambda-is-dense-in-SO-lambda-0} by $C^\infty$,
because $f\in SO_\infty$ is bounded and its derivatives $f^{(k)}$ are bounded
for all $k\in\N$ in view of $\lim_{x\to\infty}(D_\infty^kf)(x)=0$.
Since $SO_\infty^\infty\subset SO_\infty^3$, this completes the proof in the
case $\lambda=\infty$.

If $\lambda\in\R$, then by \cite[Corollary~2.2]{KILH13a}, the mapping
$Ta=a\circ\beta_\lambda$, where $\beta_\lambda:\dR\to\dR$ is defined by
\begin{equation}\label{eq:SO3-lambda-is-dense-in-SO-lambda-1}
\beta_\lambda(x)=\frac{\lambda x-1}{x+\lambda},
\end{equation}
is an isometric isomorphism of the algebra $SO_\lambda$ onto the algebra
$SO_\infty$. Hence each function $a\in SO_\lambda$ can be approximated in
the norm of $SO_\lambda$ by functions $c_n=b_n\circ\beta_\lambda^{-1}$, where
$b_n\in SO_\infty^\infty$ for $n\in\N$ and
\begin{equation}\label{eq:SO3-lambda-is-dense-in-SO-lambda-2}
\beta_\lambda^{-1}(y)=\frac{\lambda y+1}{\lambda-y}=x,
\quad x,y\in\dR.
\end{equation}
It remains to show that $c_n\in SO_\lambda^3$. Taking into account
\eqref{eq:SO3-lambda-is-dense-in-SO-lambda-1}--%
\eqref{eq:SO3-lambda-is-dense-in-SO-lambda-2}, we obtain for
$y=\beta_\lambda(x)\in\R\setminus\{\lambda\}$ and
$x=\beta_\lambda^{-1}(y)\in\R$:
\begin{align}
(D_\lambda c_n)(y)
=&
b_n'\left(\beta_\lambda^{-1}(y)\right)\frac{\lambda^2+1}{y-\lambda}
=
-b_n'(x)(x+\lambda),
\label{eq:SO3-lambda-is-dense-in-SO-lambda-3}
\\
(D_\lambda^2c_n)(y)
=&
b_n''\left(\beta_\lambda^{-1}(y)\right)\frac{(\lambda^2+1)^2}{y-\lambda}
-
b_n'\left(\beta_\lambda^{-1}(y)\right)\frac{\lambda^2+1}{y-\lambda}
\nonumber
\\
=&
-b_n''(x)(x+\lambda)(\lambda^2+1)
+b_n'(x)(x+\lambda),
\label{eq:SO3-lambda-is-dense-in-SO-lambda-4}
\end{align}
\begin{align}
(D_\lambda^3c_n)(y)
=&
b_n'''\left(\beta_\lambda^{-1}(y)\right)\frac{(\lambda^2+1)^3}{y-\lambda}
-
2b_n''\left(\beta_\lambda^{-1}(y)\right)\frac{(\lambda^2+1)^2}{y-\lambda}
\nonumber
\\
&+
b_n'\left(\beta_\lambda^{-1}(y)\right)\frac{\lambda^2+1}{y-\lambda}
\nonumber
\\
=&
-b_n'''(x)(x+\lambda)(\lambda^2+1)^2
+2b_n''(x)(x+\lambda)(\lambda^2+1)
\nonumber
\\
&-b_n'(x)(x+\lambda).
\label{eq:SO3-lambda-is-dense-in-SO-lambda-5}
\end{align}
Since
\[
\lim_{x\to\infty}(D_\infty^kb_n)(x)=0
\quad\mbox{for}\quad
k\in\{1,2,3\},
\]
we see that
\begin{equation}\label{eq:SO3-lambda-is-dense-in-SO-lambda-6}
\lim_{x\to\infty}x^kb_n^{(k)}(x)=0
\quad\mbox{for}\quad
k\in\{1,2,3\}.
\end{equation}
It follows from \eqref{eq:SO3-lambda-is-dense-in-SO-lambda-3}--%
\eqref{eq:SO3-lambda-is-dense-in-SO-lambda-6} that
\[
\lim_{y\to\lambda}(D_\lambda^kc_n)(y)=0
\quad\mbox{for}\quad
k\in\{1,2,3\}.
\]
Hence $c_n\in SO_\lambda^3$ for all $n\in\N$, which completes the proof.
\end{proof}
\subsection{Slowly oscillating Fourier multipliers}
\label{sec:SO-multipliers}
The following result leads us to the definition of slowly oscillating
Fourier multipliers.
\begin{theorem}[{\cite[Theorem~2.5]{K15c}}]
\label{th:boundedness-convolution-SO}
Let $X(\R)$ be a separable Banach function space such that the
Hardy-Littlewood maximal operator $\cM$ is bounded on $X(\R)$ and on its
associate space $X'(\R)$. If $\lambda\in\dR$ and $a\in SO_\lambda^3$, then
the convolution operator $W^0(a)$ is bounded on the space $X(\R)$ and
\begin{equation}\label{eq:boundedness-convolution-SO}
\|W^0(a)\|_{\cB(X(\R))}
\le
c_{X}\|a\|_{SO_\lambda^3},
\end{equation}
where $c_{X}$ is a positive constant depending only on $X(\R)$.
\end{theorem}
Let $SO_{\lambda,X(\R)}$ denote the closure of $SO_\lambda^3$ in the norm of
$\cM_{X(\R)}$. Further, let $SO_{X(\R)}^\diamond$ be the smallest Banach
subalgebra of $\cM_{X(\R)}$ that contains all the Banach algebras
$SO_{\lambda,X(\R)}$ for $\lambda\in\dR$. The functions in
$SO_{X(\R)}^\diamond$ will be called slowly oscillating Fourier multipliers.
\begin{lemma}\label{le:SO-continuous-embedding}
Let $X(\R)$ be a separable Banach function space such that the
Hardy-Littlewood maximal operator $\cM$ is bounded on $X(\R)$ and on its
associate space $X'(\R)$. Then
\[
SO_{X(\R)}^\diamond\subset SO_{L^2(\R)}^\diamond=SO^\diamond.
\]
\end{lemma}
\begin{proof}
The continuous embedding
$SO_{X(\R)}^\diamond\subset SO_{L^2(\R)}^\diamond$
(with the embedding constant one) follows immediately from
Theorem~\ref{th:continuous-embedding} and the definitions
of the Banach algebras $SO_{X(\R)}^\diamond$ and $SO_{L^2(\R)}^\diamond$.
It is clear that $SO_{L^2(\R)}^\diamond \subset SO^\diamond$. The embedding
$SO^\diamond\subset SO_{L^2(\R)}^\diamond$ follows from
Lemma~\ref{le:SO3-lambda-is-dense-in-SO-lambda}.
\end{proof}
\section{Maximal ideal spaces of the algebras \boldmath{$SO^\diamond$} and
\boldmath{$SO_{X(\R)}^\diamond$}}
\label{sec:maximal-ideal-space}
\subsection{Extensions of multiplicative linear functionals on
\boldmath{$C^*$}-algebras}
For a $C^*$-algebra (or, more generally, a Banach algebra) $\fA$
with unit $e$ and an element $a\in\fA$, let $\spe_\fA(a)$ denote the spectrum
of $a$ in $\fA$. Recall that an element $a$ of a $C^*$-algebra $\fA$ is said
to be positive if it is self-adjoint and $\spe_\fA(a)\subset[0,\infty)$. A
linear functional $\phi$ on $\fA$ is said to be a state if $\phi(a)\ge 0$
for all positive elements $a\in\fA$ and $\phi(e)=1$. The set of all states
of $\fA$ is denoted by $\fS(\fA)$. The extreme points of $\fS(\fA)$ are called
pure states of $\fA$ (see, e.g., \cite[Section~4.3]{KR97}).

Following \cite[p.~304]{A79}, for a state $\phi$, let
\[
\mathcal{G}_\phi(\fA):=\{a\in\fA\ :\ |\phi(a)|=\|a\|_\fA=1\}
\]
and let $\mathcal{G}_\phi^+(\fA)$ denote the set of all positive elements of
$\mathcal{G}_\phi(\fA)$. Let $\fA$ and $\fB$ be $C^*$-algebras such that
$e\in\fB\subset\fA$. Let $\phi$ be a state of $\fB$. Following
\cite[p.~310]{A79}, we say that $\fA$ is $\fB$-compressible modulo $\phi$
if for each $x\in\fA$ and each $\eps>0$ there is $b\in\mathcal{G}_\phi^+(\fB)$
and $y\in\fB$ such that $\|bxb-y\|_\fA<\eps$.

Since a nonzero linear functional on a commutative $C^*$-algebra is a pure
state if and only if it is multiplicative (see, e.g.,
\cite[Proposition~4.4.1]{KR97}), we immediately get the following
lemma from \cite[Theorem~3.2]{A79}.
\begin{lemma}\label{le:extension-characters}
Let $\fB$ be a $C^*$-subalgebra of a commutative $C^*$-algebra $\fA$.
A nonzero multiplicative linear functional $\phi$ on $\fB$ admits a unique
extension to a multiplicative linear functional $\phi'$ on $\fA$ if and
only if $\fA$ is $\fB$-compressible modulo $\phi$.
\end{lemma}
\subsection{Family of positive elements}
For $t\in\dR$ and $\omega>0$, let $\psi_{t,\omega}$ be a real-valued
function in $C(\dR)$ such that $0\le \psi_{t,\omega}(x)\le 1$
for all $x\in\R$. Assume that for $t\in\R$,
\[
\psi_{t,\omega}(s)=1
\quad\mbox{if}\quad
s\in(t-\omega,t+\omega),
\quad
\psi_{t,\omega}(s)=0
\quad\mbox{if}\quad
s\in\R\setminus(t-2\omega,t+2\omega),
\]
and for $t=\infty$,
\[
\psi_{\infty,\omega}(s)=1
\quad\mbox{if}\quad
s\in\R\setminus(-2\omega,2\omega),
\quad
\psi_{\infty,\omega}(s)=0
\quad\mbox{if}\quad
s\in(-\omega,\omega).
\]

Let $M(\fA)$ denote the maximal ideal space of a commutative
Banach algebra $\fA$.
\begin{lemma}\label{le:positive-elements}
For $t\in\dR$ and $\omega>0$, the function $\psi_{t,\omega}$ is a positive
element of the $C^*$-algebras $C(\dR)$, $SO_t$, and $SO^\diamond$.
\end{lemma}
\begin{proof}
Since $M(C(\dR))=\dR$, it follows from the Gelfand theorem (see, e.g.,
\cite[Theorem~2.1.3]{RSS11}) that $\spe_{C(\dR)}(\psi_{t,\omega})=[0,1]$
for all $t\in\dR$ and all $\omega>0$. Since
$C(\dR)\subset SO_t\subset SO^\diamond$, we conclude that
the functions $\psi_{t,\omega}$ for $t\in\dR$ and $\omega>0$ are positive
elements of the $C^*$-algebras $C(\dR)$, $SO_t$, and $SO^\diamond$
because their spectra in each of these algebras coincide with
$[0,1]$ in view of \cite[Proposition~4.1.5]{KR97}.
\end{proof}
\subsection{Maximal ideal space of the \boldmath{$C^*$}-algebra
\boldmath{$SO^\diamond$}}
If $\fB$ is a Banach subalgebra of $\fA$ and
$\lambda\in M(\fB)$, then the set
\[
M_\lambda(\fA):=\{\xi\in M(\fA):\xi|_{\fB}=\lambda\}
\]
is called the fiber of $M(\fA)$ over $\lambda\in M(\fB)$.
Hence for every Banach algebra $\Phi\subset L^\infty(\R)$ with
$M(C(\dR)\cap\Phi)=\dR$ and every $t\in\dR$, the fiber $M_t(\Phi)$ is
the set of all multiplicative linear functionals (characters) on $\Phi$ that
annihilate the set $\{f\in C(\dR)\cap\Phi:f(t)=0\}$. As usual, for all
$a\in\Phi$ and all $\xi\in M(\Phi)$, we put $a(\xi):=\xi(a)$. We will
frequently identify the points $t\in\dR$ with the evaluation functionals
$\delta_t$ defined by
\[
\delta_t(f)=f(t)
\quad\mbox{for}\quad f\in C(\dR),\quad t\in\dR.
\]
\begin{lemma}\label{le:fibers-SOt-SO-diamond}
For every point $t\in\dR$, the fibers $M_t(SO_t)$ and $M_t(SO^\diamond)$ can
be identified as sets:
\begin{equation}\label{eq:fibers-SOt-SO-diamond-1}
M_t(SO_t)=M_t(SO^\diamond).
\end{equation}
\end{lemma}
\begin{proof}
Since $C(\dR)\subset SO_t\subset SO^\diamond$, by the restriction of a
multiplicative linear functional defined on a bigger algebra to a smaller
algebra, we have
\begin{equation}\label{eq:fibers-SOt-SO-diamond-2}
M(SO^\diamond)\subset M(SO_t)\subset M(C(\dR)),
\quad t\in\dR.
\end{equation}
Since
\[
M(\Phi)=\bigcup_{t\in\dR}M_t(\Phi)
\quad\mbox{for}\quad
\Phi\in\{SO^\diamond,SO_\lambda:\lambda\in\dR\},
\]
where
\begin{equation}\label{eq:fibers-SOt-SO-diamond-3}
M_t(\Phi)=\{\zeta\in M(\Phi)\ : \ \zeta|_{C(\dR)}=\delta_t\},
\quad t\in\dR,
\end{equation}
it follows from \eqref{eq:fibers-SOt-SO-diamond-2}
and \eqref{eq:fibers-SOt-SO-diamond-3} that
\begin{equation}\label{eq:fibers-SOt-SO-diamond-4}
M_t(SO^\diamond)\subset M_t(SO_t),
\quad
t\in\R.
\end{equation}

Now fix $t\in\dR$ and a multiplicative linear functional $\eta\in M_t(SO_t)$.
Let us show that the $C^*$-algebra $SO^\diamond$ is $SO_t$-compressible
modulo $\eta$. Take $\eps>0$. By the definition of $SO^\diamond$, for a
function $x\in SO^\diamond$, there are a finite set $F\in\dR$ and
a finite set $\{x_\lambda\in SO_\lambda:\ \lambda\in F\}$ such that
\[
\left\|x-\sum_{\lambda\in F}x_\lambda\right\|_{L^\infty(\R)}<\eps.
\]

If $t\ne\infty$, take $\omega$ such that
\[
0<\omega<\frac{1}{2}\min_{\lambda\in F\setminus\{t\}}|\lambda-t|
\]
and $b:=\psi_{t,\omega}$. Then
\begin{equation}\label{eq:fibers-SOt-SO-diamond-5}
y:=b\left(\sum_{\lambda\in F}x_\lambda\right)b
\end{equation}
is equal to zero outside the interval $(t-2\omega,t+2\omega)$. Therefore,
$y\in SO_t$.

If $t=\infty$, take $\omega$ such that
\[
\omega>\max_{\lambda\in F\setminus\{\infty\}}|\lambda|
\]
and $b:=\psi_{\infty,\omega}$. Then the function $y$ defined by
\eqref{eq:fibers-SOt-SO-diamond-5} is equal to zero on $(-\omega,\omega)$
and $y\in SO_\infty$.

For $t\in\dR$, we have
\[
\|bxb-y\|_{L^\infty(\R)}
=
\left\|b\left(x-\sum_{\lambda\in F}x_\lambda\right)b\right\|_{L^\infty(\R)}
\le
\left\|x-\sum_{\lambda\in F}x_\lambda\right\|_{L^\infty(\R)}<\eps.
\]
Since $b$ is a positive element of $SO_t$ in view of
Lemma~\ref{le:positive-elements}, we have $b\in\mathcal{G}_\eta^+(SO_t)$,
which completes the proof of the fact that $SO^\diamond$ is
$SO_t$-compressible modulo the multiplicative linear functional
$\eta\in M_t(SO_t)$.

In view of Lemma~\ref{le:extension-characters}, there exists a unique
extension $\eta'$ of the multiplicative linear functional $\eta$ to the whole
algebra $SO^\diamond$. By the definition of the fiber $M_t(SO^\diamond)$, we
have $\eta'\in M_t(SO^\diamond)$. Thus, we can identify $M_t(SO_t)$ with a
subset of $M_t(SO^\diamond)$:
\begin{equation}\label{eq:fibers-SOt-SO-diamond-6}
M_t(SO_t)\subset M_t(SO^\diamond).
\end{equation}
Combining \eqref{eq:fibers-SOt-SO-diamond-4} and
\eqref{eq:fibers-SOt-SO-diamond-6}, we arrive at
\eqref{eq:fibers-SOt-SO-diamond-1}.
\end{proof}
\begin{corollary}
The maximal ideal space of the commutative $C^*$-algebra $SO^\diamond$
can be identified with the set
\[
\bigcup_{t\in\dR}M_t(SO_t).
\]
\end{corollary}
\subsection{Extensions of multiplicative linear functionals on Banach algebras}
The following theorem in a slightly different form is contained in
\cite[Theorem~2.1.1]{S07} and \cite[Theorem~3.10]{SM86}. For the convenience
of readers, we give its proof here.
\begin{theorem}\label{th:Simonenko}
Let $\fA,\fB,\fC$ be commutative unital Banach algebras with common unit
and homomorphic imbeddings $\fA\subset\fB\subset\fC$, where $\fA$ is dense
in $\fB$. If for each functional $\varphi\in M(\fA)$ there exists a unique
extension $\varphi'\in M(\fC)$, then for every functional $\psi\in M(\fB)$
there exists a unique extension $\psi'\in M(\fC)$.
\end{theorem}
\begin{proof}
Let $\psi\in M(\fB)$. Then $\psi_1:=\psi|_{\fA}\in M(\fA)$. By the
hypotheses, there exists a unique extension $\psi_3:=(\psi_1)'\in M(\fC)$.
Then $\psi_1(a)=\psi(a)=\psi_3(a)$ for all $a\in\fA$. Let $\psi_2:=
\psi_3|_\fB\in M(\fB)$. Since $\fA\subset\fB$, it follows that
\begin{equation}\label{eq:Simonenko-1}
\psi(a)=\psi_2(a)\quad\mbox{for all}\quad a\in\fA.
\end{equation}
On the other hand, functionals $\psi,\psi_2\in M(\fB)$ are continuous on
$\fB$ (see, e.g., \cite[Lemma~2.1.5]{K09}). Since $\fA$ is dense in $\fB$,
for every $b\in\fB$ there exists a sequence $\{a_n\}_{n\in\N}\subset\fA$
such that $\|a_n-b\|_\fB\to 0$ as $n\to\infty$.
It follows from this observation and \eqref{eq:Simonenko-1} that for every
$b\in\fB$,
\[
\psi(b)=\lim_{n\to\infty}\psi(a_n)=\lim_{n\to\infty}\psi_2(a_n)=\psi_2(b)
=\psi_3(b).
\]
Thus $\psi_3\in M(\fC)$ is an extension of $\psi$. This extension is unique
by construction.
\end{proof}
\subsection{Maximal ideal space of the Banach algebras
\boldmath{$SO_{t,X(\R)}$}}
We start with the following refinement of \cite[Lemma~3.4]{KILH13a}.
\begin{lemma}\label{le:extension-from-SOt3-to-SOt}
Let $t\in\dR$. Then for each functional $\varphi\in M(SO_t^3)$ there exists
a unique extension $\varphi'\in M(SO_t)$.
\end{lemma}
The density of $SO_t^3$ in the Banach algebra $SO_t$ essentially used in the
proof of \cite[Lemma~3.4]{KILH13a} is justified in
Lemma~\ref{le:SO3-lambda-is-dense-in-SO-lambda}. Note that the uniqueness of
an extension was not explicitly mentioned in \cite[Lemma~3.4]{KILH13a}.
However, since $M(SO_t^3)$ and $M(SO_t)$ are Hausdorff spaces (see, e.g.,
\cite[Theorem~2.2.3]{K09}), the uniqueness of an extension constructed in
the proof of \cite[Lemma~3.4]{KILH13a} is a consequence of a standard fact
from general topology (see, e.g,, \cite[Theorem~IV.2(b)]{RS80}).

The following lemma is analogous to \cite[Lemma~3.5]{KILH13a}.
\begin{lemma}\label{le:equality-of-maximal-ideal-spaces-local}
Let $X(\R)$ be a separable Banach function space such that the Hardy-Littlewood
maximal operator $\cM$ is bounded on the space $X(\R)$ and on its associate
space $X'(\R)$. If $t\in\dR$, then the maximal ideal spaces of the
$C^*$-algebra $SO_t$ and the Banach algebra $SO_{t,X(\R)}$ can be identified
as sets:
\begin{equation}\label{eq:equality-of-maximal-ideal-spaces-local-1}
M(SO_t)=M(SO_{t,X(\R)}).
\end{equation}
\end{lemma}
\begin{proof}
It follows from Theorem~\ref{th:continuous-embedding} that
$SO_t^3\subset SO_{t,X(\R)}\subset SO_t$, where the
imbeddings are homomorphic. By the definition of the algebra
$SO_{t,X(\R)}$, the algebra $SO_t^3$ is dense in $SO_{t,X(\R)}$ with
respect to the norm of $\cM_{X(\R)}$. Taking into account these observations
and Lemma~\ref{le:extension-from-SOt3-to-SOt}, we see that the commutative
Banach algebras
\[
\fA=SO_t^3,\quad \fB=SO_{t,X(\R)},\quad \fC=SO_t
\]
satisfy all the conditions of Theorem~\ref{th:Simonenko}. By this theorem,
every multiplicative linear functional on $SO_{t,X(\R)}$ admits a unique
extension to a multiplicative linear functional on $SO_t$. Hence
we can identify $M(SO_{t,X(\R)})$ with a subset of $M(SO_t)$:
\begin{equation}\label{eq:equality-of-maximal-ideal-spaces-local-2}
M(SO_{t,X(\R)})\subset M(SO_t).
\end{equation}
On the other hand, since $SO_{t,X(\R)}\subset SO_t$, by the restriction of a
multiplicative linear functional defined on a bigger algebra to a smaller
algebra, we have
\begin{equation}\label{eq:equality-of-maximal-ideal-spaces-local-3}
M(SO_t)\subset M(SO_{t,X(\R)}).
\end{equation}
Combining inclusions \eqref{eq:equality-of-maximal-ideal-spaces-local-2}
and \eqref{eq:equality-of-maximal-ideal-spaces-local-3},
we immediately arrive at
\eqref{eq:equality-of-maximal-ideal-spaces-local-1}.
\end{proof}
The next lemma is analogous to Lemma~\ref{le:fibers-SOt-SO-diamond}.
\begin{lemma}\label{le:fibers-multipliers-SOt-SO-diamond}
Let $X(\R)$ be a separable Banach function space such that the Hardy-Littlewood
maximal operator $\cM$ is bounded on the space $X(\R)$ and on its associate
space $X'(\R)$. Then, for every point $t\in\dR$, the fibers
$M_t(SO_{t,X(\R)})$ and $M_t(SO_{X(\R)}^\diamond)$ can be identified as sets:
\begin{equation}\label{eq:fibers-multipliers-SOt-SO-diamond-1}
M_t(SO_{t,X(\R)})=M_t(SO_{X(\R)}^\diamond).
\end{equation}
\end{lemma}
\begin{proof}
Since $SO_{t,X(\R)}\subset SO_{X(\R)}^\diamond$ for every $t\in\dR$, we
conclude by the restriction of a multiplicative linear functional defined on
the bigger algebra to the smaller algebra that
$M(SO_{X(\R)}^\diamond)\subset M(SO_{t,X(\R)})$. Hence
\begin{equation}\label{eq:fibers-multipliers-SOt-SO-diamond-2}
M_t(SO_{X(\R)}^\diamond)\subset M_t(SO_{t,X(\R)}).
\end{equation}
On the other hand, in view of
Lemma~\ref{le:equality-of-maximal-ideal-spaces-local}, any
multiplicative linear functional $\xi\in M_t(SO_{t,X(\R)})$ admits a
unique extension $\xi'\in M(SO_t)$. Moreover,
$\xi'$ belongs to $M_t(SO_t)$ as well.
By Lemma~\ref{le:fibers-SOt-SO-diamond}, the functional
$\xi'\in M_t(SO_t)$ admits a unique extension $\xi''\in M_t(SO^\diamond)$.
It is clear that the restriction of $\xi''$ to $SO_{X(\R)}^\diamond$
belongs to $M_t(SO_{X(\R)}^\diamond)$. Thus $M_t(SO_{t,X(\R)})$
can be identified with a subset of $M_t(SO_{X(\R)}^\diamond)$:
\begin{equation}\label{eq:fibers-multipliers-SOt-SO-diamond-3}
M_t(SO_{t,X(\R)})\subset M_t(SO_{X(\R)}^\diamond).
\end{equation}
Combining \eqref{eq:fibers-multipliers-SOt-SO-diamond-2} and
\eqref{eq:fibers-multipliers-SOt-SO-diamond-3}, we arrive at
\eqref{eq:fibers-multipliers-SOt-SO-diamond-1}.
\end{proof}
\subsection{Maximal ideal space of the Banach algebra
\boldmath{$SO_{X(\R)}^\diamond$}}
Now we are in a position to prove that the maximal ideal spaces
of the commutative Banach algebra $SO_{X(\R)}^\diamond$ and
the $C^*$-algebra $SO^\diamond$ can be identified as sets.
\begin{theorem}
Let $X(\R)$ be a separable Banach function space such that the Hardy-Littlewood
maximal operator $\cM$ is bounded on the space $X(\R)$ and on its associate
space $X'(\R)$. Then the maximal ideal space of the Banach algebra
$SO_{X(\R)}^\diamond$ can be identified with the maximal ideal space
of the $C^*$-algebra $SO^\diamond$:
\[
M(SO_{X(\R)}^\diamond)=M(SO^\diamond).
\]
\end{theorem}
\begin{proof}
It follows from Lemmas~\ref{le:fibers-multipliers-SOt-SO-diamond},
\ref{le:equality-of-maximal-ideal-spaces-local} and
\ref{le:fibers-SOt-SO-diamond} that for every $t\in\dR$,
\[
M_t(SO_{X(\R)}^\diamond)=M_t(SO_{t,X(\R)})=M_t(SO_t)=M_t(SO^\diamond).
\]
Hence
\[
M(SO_{X(\R)}^\diamond)
=
\bigcup_{t\in\dR}M_t(SO_{X(\R)}^\diamond)
=
\bigcup_{t\in\dR}M_t(SO^\diamond)
=M(SO^\diamond),
\]
which completes the proof.
\end{proof}
\section{Maximal ideal spaces of the Calkin images of the Banach algebras
\boldmath{$\mathcal{MO}(\Phi)$} and \boldmath{$\mathcal{CO}(\Psi)$}}
\label{sec:maximal-ideal-spaces-Calkin-images}
\subsection{Maximal ideal space of the algebra
\boldmath{$\mathcal{MO}^\pi(\Phi)$}}
We start with the following known result \cite[Theorem~2.4]{HKK06}
(see also \cite[Theorem~3.1]{FKK-FS12}).
\begin{theorem}\label{th:noncompactness-multiplication}
Let $X(\R)$ be a separable Banach function space and $a\in L^\infty(\R)$. Then
the  multiplication operator $aI$ is compact on the space $X(\R)$ if and only
if $a=0$ almost everywhere on $\R$.
\end{theorem}
The next theorem says that one can identify the maximal ideal spaces
of the algebras $\mathcal{MO}^\pi(\Phi)$ and $\Phi$ for an arbitrary
unital $C^*$-subalgebra of $L^\infty(\R)$.
\begin{theorem}\label{th:maximal-ideal-space-quotient-algebra-MO}
Let $X(\R)$ be a separable Banach function space.
If $\Phi$ is a unital $C^*$-subalagebra of $L^\infty(\R)$, then
the maximal ideal spaces of the commutative Banach algebra
$\mathcal{MO}^\pi(\Phi)$ and the commutative $C^*$-algebra $\Phi$
are homeomorphic:
\[
M(\mathcal{MO}^\pi(\Phi))=M(\Phi).
\]
\end{theorem}
\begin{proof}
Consider the mapping $F:\Phi\to\mathcal{MO}^\pi(\Phi)$ defined
by $F(a)=[aI]^\pi$
for every $a\in\Phi$. It is clear that this mapping is surjective. If
$[aI]^\pi=[bI]^\pi$ for some $a,b\in\Phi$, then
$(a-b)I\in\cK(X(\R))$. It follows from
Theorem~\ref{th:noncompactness-multiplication} that $a=b$ a.e. on $\R$.
This implies that the mapping $F$ is injective. Thus,
$F:\Phi\to\mathcal{MO}^\pi(\Phi)$
is an algebraic isomorphism of commutative Banach algebras. It follows from
\cite[Lemma~2.2.12]{K09} that the maximal ideal spaces
$M(\mathcal{MO}^\pi(\Phi))$ and $M(\Phi)$ are homeomorphic.
\end{proof}
\subsection{Maximal ideal space of the algebra
\boldmath{$\mathcal{CO}^\pi(\Psi)$}}
The following analogue of Theorem~\ref{th:noncompactness-multiplication}
for Fourier convolution operators was obtained recently by the authors
\cite[Theorem~1.1]{FKK-AFA}.
\begin{theorem}
\label{th:noncompactness-convolution}
Let $X(\R)$ be a separable Banach function space such that the
Hardy-Littlewood maximal operator $\cM$ is bounded on $X(\R)$ and on its
associate space $X'(\R)$. Suppose that $b\in\cM_{X(\R)}$. Then the Fourier
convolution operator $W^0(a)$ is compact on the space $X(\R)$ if and only
if $b=0$ almost everywhere on $\R$.
\end{theorem}
The next theorem is an analogue of
Theorem~\ref{th:maximal-ideal-space-quotient-algebra-MO}
for Fourier multipliers.
\begin{theorem}\label{th:maximal-ideal-space-quotient-algebra-CO}
Let $X(\R)$ be a separable Banach function space such that the Hardy-Littlewood
maximal operator $\cM$ is bounded on $X(\R)$ and on its associate space
$X'(\R)$. If $\Psi$ is a unital Banach subalagebra of
$\cM_{X(\R)}$, then the maximal ideal spaces of the commutative
Banach algebras $\mathcal{CO}^\pi(\Psi)$ and $\Psi$ are
homeomorphic:
\[
M(\mathcal{CO}^\pi(\Psi))=M(\Psi).
\]
\end{theorem}
\begin{proof}
The proof is analogous to the proof of
Theorem~\ref{th:maximal-ideal-space-quotient-algebra-MO}.
Consider the mapping $F:\Psi\to\mathcal{CO}^\pi(\Psi)$ defined
by $F(a)=[W^0(a)]^\pi$ for every $a\in\Psi$. It is obvious that this mapping
is surjective. If
$[W^0(a)]^\pi=[W^0(b)]^\pi$ for some $a,b\in\Psi$, then
$W^0(a-b)=W^0(a)-W^0(b)\in\cK(X(\R))$. By
Theorem~\ref{th:noncompactness-convolution},
we conclude that $a=b$ a.e. on $\R$. Therefore, the mapping $F$ is injective.
Thus, $F:\Psi\to\mathcal{CO}^\pi(\Psi)$ is an algebraic
isomorphism of commutative
Banach algebras. In this case it follows from \cite[Lemma~2.2.12]{K09}
that the maximal ideal spaces $M(\mathcal{CO}^\pi(\Psi))$ and
$M(\Psi)$ are homeomorphic.
\end{proof}
\section{Applications of the method of limit operators}
\label{sec:proofs}
\subsection{Known {result} about limit operators on Banach
function spaces}
Let $X(\R)$ be a Banach function space. For a sequence of operators
$\{A_n\}_{n\in\N}\subset\cB(X(\R))$, let
\[
\slim A_n
\]
denote the strong limit of this sequence, if it exists.
For $\lambda,x\in\R$, consider the function
\[
e_\lambda(x):=e^{i\lambda x}.
\]

Let $T\in\cB(X(\R))$ and let $h=\{h_n\}_{n\in\N}$ be a sequence of numbers
$h_n>0$ such that $h_n\to +\infty$ as $n\to\infty$. The strong limit
\[
T_{h}:=\operatornamewithlimits{s-\lim}_{n\to\infty}
e_{h_n}Te_{h_n}^{-1}I
\]
is called the limit operator of $T$ related to the sequence
$h=\{h_n\}_{n\in\N}$, if it exists.

{
In our previous paper \cite{FKK-FS12} we calculated the limit
operators for all compact operators.
}
\begin{lemma}[{\cite[Lemma~3.2]{FKK-FS12}}]
\label{le:LO-compact}
Let $X(\R)$ be a separable Banach function space and $K$ be a compact operator
on $X(\R)$. Then for every sequence $\{h_n\}_{n\in\N}$ of positive numbers
satisfying $h_n\to+\infty$ as $n\to\infty$, one has
\[
\slim e_{h_n}Ke_{h_n}^{-1}I=0
\]
on the space $X(\R)$.
\end{lemma}
\subsection{Limit operators for Fourier convolution operators with symbols
in the algebra \boldmath{$SO_{X(\R)}^\diamond$}}
Now we will calculate the limit operators for the Fourier convolution
operator with a slowly oscillating symbol.
\begin{theorem}\label{th:LO-convolution-SO}
Let $X(\R)$ be a separable Banach function space such that the Hardy-Littlewood
maximal operator $\cM$ is bounded on the space $X(\R)$ and on its associate
space $X'(\R)$. If $b\in SO_{X(\R)}^\diamond$, then for every
$\xi\in M_\infty(SO^\diamond)$ there exists a sequence
$\{h_n\}_{n\in\N}$ of positive numbers such that
$h_n\to+\infty$ as $n\to\infty$ and
\begin{equation}\label{eq:LO-convolution-SO-1}
\slim e_{h_n}W^0(b)e_{h_n}^{-1}I=b(\xi)I
\end{equation}
on the space $X(\R)$.
\end{theorem}
\begin{proof}
This statement is proved by analogy with \cite[Lemma~5.1]{KILH13a}.
In view of Lemma~\ref{le:SO-continuous-embedding},
$SO_{X(\R)}^\diamond\subset SO^\diamond$. Therefore every
$\xi\in M_\infty(SO^\diamond)$ is a multiplicative linear functional
on $SO_{X(\R)}^\diamond$, that is, $b(\xi)$ is well defined. By the definition
of $SO_{X(\R)}^\diamond$, if $b\in SO_{X(\R)}^\diamond$, then there is a
sequence
\[
b_m=\sum_{\lambda\in F_m}b_{m,\lambda},\quad m\in\N,
\]
where $F_m\subset\dR$ are finite sets and $b_{m,\lambda}\in SO_\lambda^3$
for $\lambda\in F_m$ and all $m\in\N$, such that
\begin{equation}\label{eq:LO-convolution-SO-2}
\lim_{m\to\infty}\|b_m-b\|_{\cM_{X(\R)}}=0.
\end{equation}
By Lemma~\ref{le:fibers-SOt-SO-diamond},
$M_\infty(SO^\diamond)=M_\infty(SO_\infty)$. Fix
$\xi\in M_\infty(SO^\diamond)=M_\infty(SO_\infty)$. Assume first that the set
\[
B_\infty:=\{b_{m,\infty}\in SO_\infty^3\ :\ m\in\N\}
\]
is not empty. Since the set $B_\infty$ is at most countable, it follows from
\cite[Corollary~3.3]{BBK04} or \cite[Proposition~3.1]{KILH13a} that there
exists a sequence $\{h_n\}_{n\in\N}$ such that $h_n\to+\infty$ as $n\to\infty$
and
\begin{equation}\label{eq:LO-convolution-SO-3}
\xi(b_{m,\infty})=\lim_{n\to\infty} b_{m,\infty}(h_n)
\quad\mbox{for all}\quad b_{m,\infty}\in B_\infty.
\end{equation}
As the functions $b_{m,\lambda}$ are continuous at $\infty$ if
$\lambda\ne\infty$, we see that
\begin{equation}\label{eq:LO-convolution-SO-4}
\xi(b_{m,\lambda})=b_{m,\lambda}(\infty)=\lim_{n\to\infty} b_{m,\lambda}(h_n)
\quad\mbox{for all}\quad \lambda\in\bigcup_{m\in\N}F_m\setminus\{\infty\}.
\end{equation}
Combining \eqref{eq:LO-convolution-SO-3} and \eqref{eq:LO-convolution-SO-4},
for every $m\in\N$, we get
\begin{align}\label{eq:LO-convolution-SO-5}
\xi(b_m)
&=
\sum_{\lambda\in F_m} \xi(b_{m,\lambda})
=
\sum_{\lambda\in F_m}
\lim_{n\to\infty}b_{m,\lambda}(h_n)
\nonumber\\
&=
\lim_{n\to\infty}\sum_{\lambda\in F_m}b_{m,\lambda}(h_n)
=
\lim_{n\to\infty}b_m(h_n).
\end{align}
If the set $B_\infty$ is empty, we can take an arbitrary sequence
$\{h_n\}_{n\in\N}$ such that $h_n\to+\infty$ as $n\to\infty$.

Let $f\in\cS_0(\R)$. Then, by a smooth version of Urysohn's lemma
(see, e.g., \cite[Proposition~6.5]{F09}), there is a function
$\psi\in C_0^\infty(\R)$ such that $0\le\psi\le 1$,
$\operatorname{supp}\cF f\subset\operatorname{supp}\psi$
and $\psi|_{\operatorname{supp}\cF f}=1$. Therefore, for all $n\in\N$,
\begin{align*}
e_{h_n}W^0(b)e_{h_n}^{-1}f-b(\xi)f
&=
W^0[b(\cdot+h_n)]f-\xi(b)f
\\
&=
\cF^{-1}[b(\cdot+h_n)-\xi(b)]\psi\cF f
\end{align*}
and
\begin{equation}\label{eq:LO-convolution-SO-6}
\big\|\big(e_{h_n}W^0(b)e_{h_n}^{-1}-b(\xi)\big)f\big\|_{{X(\R)}}
\le
\big\|\big[b(\cdot+h_n)-\xi(b)\big]\psi\|_{\cM_{X(\R)}}\|f\|_{X(\R)}.
\end{equation}
Since $\cM_{X(\R)}$ is translation-invariant and $\xi\in M_\infty(SO^\diamond)$
is a multiplicative linear functional on $SO_{X(\R)}^\diamond$, we infer for
all $m,n\in\N$ that
\begin{align}
\big\|\big[b(\cdot+h_n)-\xi(b)\big]\psi\|_{\cM_{X(\R)}}
\le&
\big\|\big[b(\cdot+h_n)-b_m(\cdot+h_n)\big]\psi\|_{\cM_{X(\R)}}
\nonumber\\
&+
\big\|\big[b_m(\cdot+h_n)-\xi(b_m)\big]\psi\|_{\cM_{X(\R)}}
\nonumber\\
&+
\big\|\xi(b_m)-\xi(b)\big]\psi\|_{\cM_{X(\R)}}
\nonumber\\
\le &
2\|b-b_m\|_{\cM_{X(\R)}}\|\psi\|_{\cM_{X(\R)}}
\nonumber\\
&+
\big\|\big[b_m(\cdot+h_n)-\xi(b_m)\big]\psi\|_{\cM_{X(\R)}}.
\label{eq:LO-convolution-SO-7}
\end{align}

Fix $\eps>0$.
By Theorem~\ref{th:Stechkin}, $\|\psi\|_{\cM_{X(\R)}}<\infty$.
It follows from \eqref{eq:LO-convolution-SO-2}
that there exists a sufficiently large number $m\in\N$
(which we fix until the end of the proof)
such that
\begin{equation}\label{eq:LO-convolution-SO-8}
2\|b-b_m\|_{\cM_{X(\R)}}\|\psi\|_{\cM_{X(\R)}}<\eps/2.
\end{equation}

Let
\[
\Lambda:=\begin{cases}
\max\limits_{\lambda\in F_m\setminus\{\infty\}}|\lambda|
&\mbox{if}\quad F_m\setminus\{\infty\}\ne\emptyset,
\\
0 &\mbox{if}\quad F_m\setminus\{\infty\}=\emptyset,
\end{cases}
\]
let $K:=\operatorname{supp}\psi$ and
\[
k:=\max\{-\inf K,\sup K\}\in[0,\infty).
\]
For $x\in K$ and $n\in\N$, let $I_n(x)$ be the segment with the endpoints $h_n$
and $x+h_n$. Then $I_n(x)\subset [h_n-k,h_n+k]$.
Since $h_n\to+\infty$ as $n\to\infty$, there exists $N_1\in\N$ such that for
all $n>N_1$, one has
\[
I_n(x)\subset[h_n-k,h_n+k]\subset(\Lambda,\infty).
\]
For all $n>N_1$, we have
\begin{align}
\left\|
\big[b_m(\cdot+h_n)-\xi(b_m)\big]\psi
\right\|_{\cM_{X(\R)}}
\le&
\left\|\big[b_m(\cdot+h_n)-b_m(h_n)\big]\psi
\right\|_{\cM_{X(\R)}}
\nonumber\\
&+
\big|b_m(h_n)-\xi(b_m)\big|\|\psi\|_{\cM_{X(\R)}},
\label{eq:LO-convolution-SO-9}
\end{align}
where the functions
$\big[b_m(\cdot+h_n)-b_m(h_n)\big]\psi$ for $n>N_1$
belong to $SO_\infty^3$ because they are three times continuously
differentiable functions of compact support.

By \eqref{eq:LO-convolution-SO-5}, there exists $N_2\in\N$ such that
$N_2\ge N_1$ and for all $n>N_2$,
\begin{equation}\label{eq:LO-convolution-SO-10}
\big|b_m(h_n)-\xi(b_m)\big|\|\psi\|_{\cM_{X(\R)}}<\eps/4.
\end{equation}
On the other hand, since
$\big[b_m(\cdot+h_n)-b_m(h_n)\big]\psi\in SO_\infty^3$
for all $n>N_2$, it follows from Theorem~\ref{th:boundedness-convolution-SO}
that there exists a constant $c_X>0$ depending only on the space $X(\R)$
such that for all $n>N_2$,
\begin{align}
&\left\|\big[b_m(\cdot+h_n)-b_m(h_n)\big]\psi
\right\|_{\cM_{X(\R)}}
\nonumber
\\
&\quad\quad\le c_X
\left\|
\big[b_m(\cdot+h_n)-b_m(h_n)\big]\psi
\right\|_{SO_\infty^3}
\nonumber
\\
&\quad\quad
=c_X
\sum_{j=0}^3\frac{1}{j!}
\left\|
D_\infty^j\left(
\big[b_m(\cdot+h_n)-b_m(h_n)\big]\psi
\right)
\right\|_{L^\infty(\R)}.
\label{eq:LO-convolution-SO-11}
\end{align}
For all $j\in\{0,1,2,3\}$, we have
\begin{align}
&
D_\infty^j\left(
\big[b_m(\cdot+h_n)-b_m(h_n)\big]\psi
\right)
\nonumber
\\
&\quad\quad
=\sum_{\nu=0}^j \binom{j}{\nu}
\left(D_\infty^\nu
\big[b_m(\cdot+h_n)-b_m(h_n)\big]
\right)
\left(D_\infty^{j-\nu}\psi\right)\!.
\label{eq:LO-convolution-SO-12}
\end{align}
It follows from the mean value theorem that
\begin{align}
\left\|
\big[b_m(\cdot+h_n)-b_m(h_n)\big]\chi_K
\right\|_{L^\infty(\R)}
&
=
\sup_{x\in K}\left|\int_{h_n}^{x+h_n}b_m'(t)\,dt\right|
\nonumber
\\
&
=
\sup_{x\in K}\left|
\int_{h_n}^{x+h_n}tb_m'(t)\,\frac{dt}{t}
\right|
\nonumber
\\
&
\le
\sup_{x\in K}\int_{I_n(x)}
\big|(D_\infty b_m)(t)\big|\,\frac{dt}{t}
\nonumber
\\
&
\le
\sup_{t\in[h_n-k,h_n+k]}\big|(D_\infty b_m)(t)\big|
\int_{h_n-k}^{h_n+k}\frac{dt}{t}
\nonumber
\\
&
\le
\ln\frac{h_n+k}{h_n-k}
\left\|D_\infty b_m\right\|_{L^\infty(\R)}.
\label{eq:LO-convolution-SO-13}
\end{align}
It is easy to see that for $x\in K$,
\begin{align}
&
\left(
D_\infty\big[b_m(\cdot+h_n)-b_m(h_n)\big]
\right)(x)
=
\frac{x}{x+h_n}
(D_\infty b_m)(x+h_n),
\label{eq:LO-convolution-SO-14}
\\
&\big(
D_\infty^2\big[b_m(\cdot+h_n)-b_m(h_n)\big]
\big)(x)
\nonumber
\\
&
\quad\quad=
\frac{x^2}{(x+h_n)^2}
(D_\infty^2 b_m)(x+h_n)
+
\frac{xh_n}{(x+h_n)^2}(D_\infty b_m)(x+h_n),
\label{eq:LO-convolution-SO-15}
\end{align}
and
\begin{align}
&\big(
D_\infty^3\big[b_m(\cdot+h_n)-b_m(h_n)\big]
\big)(x)
\nonumber\\
&\quad\quad=
\frac{x^3}{(x+h_n)^3}
(D_\infty^3 b_m)(x+h_n)
+
\frac{3x^2h_n}{(x+h_n)^3}
(D_\infty^2 b_m)(x+h_n)
\nonumber\\
&\quad\quad\quad
+
\frac{xh_n^2-x^2h_n}{(x+h_n)^3}
(D_\infty b_m)(x+h_n).
\label{eq:LO-convolution-SO-16}
\end{align}
It follows from \eqref{eq:LO-convolution-SO-14}--%
\eqref{eq:LO-convolution-SO-16}
that for all $n>N_2$,
\begin{align}
&
\left\|\big(
D_\infty\big[b_m(\cdot+h_n)-b_m(h_n)\big]
\big)\chi_K\right\|_{L^\infty(\R)}
\le
\frac{k}{h_n-k}
\|D_\infty b_m\|_{L^\infty(\R)},
\label{eq:LO-convolution-SO-17}
\\
&
\left\|\big(
D_\infty^2\big[b_m(\cdot+h_n)-b_m(h_n)\big]
\big)\chi_K\right\|_{L^\infty(\R)}
\nonumber
\\
&\quad\quad
\le
\frac{k^2}{(h_n-k)^2}
\left\|D_\infty^2 b_m\right\|_{L^\infty(\R)}
+
\frac{kh_n}{(h_n-k)^2}
\left\|D_\infty b_m\right\|_{L^\infty(\R)},
\label{eq:LO-convolution-SO-18}
\end{align}
and
\begin{align}
&
\left\|\big(
D_\infty^3\big[b_m(\cdot+h_n)-b_m(h_n)\big]
\big)\chi_K\right\|_{L^\infty(\R)}
\nonumber\\
&\quad\quad
\le
\frac{k^3}{(h_n-k)^3}
\left\|D_\infty^3 b_m\right\|_{L^\infty(\R)}
+
\frac{3k^2h_n}{(h_n-k)^3}
\left\|D_\infty^2 b_m\right\|_{L^\infty(\R)}
\nonumber\\
&\quad\quad\quad
+
\frac{kh_n^2+k^2h_n}{(h_n-k)^3}
\left\|D_\infty b_m\right\|_{L^\infty(\R)}.
\label{eq:LO-convolution-SO-19}
\end{align}
Since
\[
\max_{j\in\{0,1,2,3\}}
\left\|D_\infty^j\psi\right\|_{L^\infty(\R)}<\infty,
\]
it follows from
\eqref{eq:LO-convolution-SO-12}--\eqref{eq:LO-convolution-SO-13} and
\eqref{eq:LO-convolution-SO-17}--\eqref{eq:LO-convolution-SO-19} that
for all $j\in\{0,1,2,3\}$,
\begin{equation}\label{eq:LO-convolution-SO-20}
\lim_{n\to\infty}
\left\|D_\infty^j
\big(\big[b_m(\cdot+h_n)-b_m(h_n)\big]\psi\big)
\right\|_{L^\infty(\R)}=0.
\end{equation}
We deduce from \eqref{eq:LO-convolution-SO-11} and
\eqref{eq:LO-convolution-SO-20} that there exists $N_3\in\N$ such that
$N_3\ge N_2$ and for all $n>N_3$,
\begin{equation}\label{eq:LO-convolution-SO-21}
\left\|
\big[b_m(\cdot+h_n)-b_m(h_n)\big]\psi\right\|_{\cM_{X(\R)}}<\eps/4.
\end{equation}
Combining \eqref{eq:LO-convolution-SO-7}--\eqref{eq:LO-convolution-SO-10}
and \eqref{eq:LO-convolution-SO-21}, we see that for every
$f\in\cS_0(\R)$ and every $\eps>0$ there exists $N_3\in\N$ such that
for all $n>N_3$,
\[
\left\|\big(e_{h_n}W^0(b)e_{h_n}^{-1}-b(\xi)\big)f\right\|_{X(\R)}
<
\eps\|f\|_{X(\R)},
\]
whence for all $f\in\cS_0(\R)$,
\[
\lim_{n\to\infty}\left\|\big(
e_{h_n}W^0(b)e_{h_n}^{-1}I-b(\xi)I
\big)f\right\|_{X(\R)}=0.
\]
Since $\cS_0(\R)$ is dense in $X(\R)$ (see Theorem~\ref{th:density-S0}),
this equality immediately implies \eqref{eq:LO-convolution-SO-1}
in view of \cite[Lemma~1.4.1(ii)]{RSS11}, which completes the proof.
\end{proof}
\subsection{Proof of Theorem~\ref{th:intersection-quotient-algebras}}
{Since} the function $e_0\equiv 1$ belongs to $\Phi$ and
$\Psi_{SO_{X(\R)}^\diamond}$, we see that the set of all constant
functions is contained in $\Phi$ and in
{$SO_{X(\R)}^\diamond$.}
Therefore
\begin{equation}\label{eq:intersection-quotient-algebras-2}
\mathcal{MO}^\pi(\C)\subset
\mathcal{MO}^\pi(\Phi)
\cap
\mathcal{CO}^\pi
{(SO_{X(\R)}^\diamond).}
\end{equation}

Let $A^\pi\in \mathcal{MO}^\pi(\Phi)\cap
\mathcal{CO}^\pi{(SO_{X(\R)}^\diamond)}$. Then
$A^\pi=[aI]^\pi=[W^0(b)]^\pi$, where $a\in\Phi$
{and $b\in SO_{X(\R)}^\diamond$}.
Therefore, there is an operator $K\in\cK(X(\R))$ such that
\begin{equation}\label{eq:intersection-quotient-algebras-3}
aI=W^0(b)+K.
\end{equation}
By Theorem~\ref{th:LO-convolution-SO}, for every $\xi\in M_\infty(SO^\diamond)$
there exists a sequence $\{h_n\}_{n\in\N}$ of positive numbers such that
$h_n\to+\infty$ as $n\to\infty$ and
\begin{equation}\label{eq:intersection-quotient-algebras-4}
\slim e_{h_n}W^0({b})e_{h_n}^{-1}I={b}(\xi)I.
\end{equation}
Equalities \eqref{eq:intersection-quotient-algebras-3}--%
\eqref{eq:intersection-quotient-algebras-4} and Lemma~\ref{le:LO-compact}
imply that
\[
aI
=
\slim e_{h_n}(aI)e_{h_n}^{-1}I
=
\slim e_{h_n}(W^0(b)+K)e_{h_n}^{-1}I
{=b(\xi)I.}
\]
Hence $[aI]^\pi=[{b}(\xi)I]^\pi\in\mathcal{MO}^\pi(\C)$ and
\begin{equation}\label{eq:intersection-quotient-algebras-5}
\mathcal{MO}^\pi(\Phi)
\cap
\mathcal{CO}^\pi
{(SO_{X(\R)}^\diamond)}
\subset\mathcal{MO}^\pi(\C).
\end{equation}
Combining \eqref{eq:intersection-quotient-algebras-2} and
\eqref{eq:intersection-quotient-algebras-5}, we arrive at
\eqref{eq:intersection-quotient-algebras-1}.
\qed
{
\subsection*{Acknowledgment}
We would like to thank the anonymous referee for pointing out a gap
in the original version of the paper. To fill in this gap, we 
strengthened the hypotheses in the main result.
}

\end{document}